\documentclass{amsart}

\usepackage{amssymb, amsfonts,amsthm,amsmath,latexsym}
\usepackage[all]{xy}

\def\bt{\begin{thm}}
\def\et{\end{thm}}
\def\bl{\begin{lem}}
\def\el{\end{lem}}
\def\bd{\begin{defi}}
\def\ed{\end{defi}}
\def\bc{\begin{cor}}
\def\ec{\end{cor}}
\def\bp{\begin{proof}}
\def\ep{\end{proof}}
\def\br{\begin{rem}}
\def\er{\end{rem}}

\theoremstyle{plain}

\newtheorem{thm}{Theorem}[section]

\newtheorem{cor}[thm]{Corollary}

\newtheorem{lem}[thm]{Lemma}
\newtheorem{proposition-principale}[thm]{Proposition principale}
\newtheorem{thm-principal}{Th\'eor\`eme principal}[section]

\theoremstyle{definition}
\newtheorem{defi}[thm]{Definition}

\newtheorem{rem}[thm]{Remark}

\numberwithin{equation}{section}

\def\C{\mathbf{C}}

\def\Z{\mathbf{Z}}

\def\k{\mathbf{k}}

\def\Aut{{\sf{Aut}}}
\def\GL{{\rm{GL}}}
\def\PGL{{\rm{PGL}}}

\def\P{\mathbb{P}}
\def\Pic{\rm{Pic}}
\def\NS{\rm{NS}}

\def\Ind{{\rm{Ind}}}

\newcommand{\bthm}{\begin{thm}}
\newcommand{\ethm}{\end{thm}}
\newcommand{\bstp}{\begin{stp}}
\newcommand{\estp}{\end{stp}}
\newcommand{\blemma}{\begin{lemma}}
\newcommand{\elemma}{\end{lemma}}
\newcommand{\bprop}{\begin{prop}}
\newcommand{\eprop}{\end{prop}}
\newcommand{\bpf}{\begin{pf}}
\newcommand{\epf}{\end{pf}}
\newcommand{\bdefn}{\begin{defn}}
\newcommand{\edefn}{\end{defn}}
\newcommand{\brk}{\begin{rmrk}}
\newcommand{\erk}{\end{rmrk}}
\newcommand{\bcrl}{\begin{crl}}
\newcommand{\ecrl}{\end{crl}}

\title{Constraints on automorphism groups of higher dimensional manifolds}
\author{Turgay Bayraktar}
\author{Serge Cantat}
\date{\today}
\address{Mathematics Department, Johns Hopkins University 21218 Maryland, USA}
\address{D\'epartement de Math\'ematiques et Applications
UMR 8553 du CNRS
Ecole Normale Sup\'erieure, 45 rue d'Ulm, Paris, France}
\email{bayraktar@jhu.edu \and serge.cantat@univ-rennes1.fr}
\begin{document}

\begin{abstract}
 In this note, we prove, for instance, that the automorphism group of a rational manifold $X$ which is obtained from $\P^k(\C)$ by a finite sequence of blow-ups along smooth centers of dimension at most $r$ with $k>2r+2$ has finite image in $GL(H^*(X,\Z))$. In particular, every holomorphic automorphism $f:X\to X$ has zero topological entropy. 
\end{abstract}

\maketitle

\section{Introduction}

\subsection{Dimensions of indeterminacy loci}
Recall that a rational map admitting a rational inverse is called birational. Birational transformations are, in general, not defined everywhere. The domain of definition of a birational map $f\colon M \to N$ is the largest Zariski-open subset on which $f$ is locally a well defined morphism. Its complement is the indeterminacy set $\Ind(f)$; its codimension 
is always larger than, or equal to, $2$. The following statement shows that the dimension 
of $\Ind(f)$ and $\Ind(f^{-1})$ can not be too small simultaneously unless $f$ is an automorphism. 
The proof of this result follows a nice argument of Nessim Sibony
concerning the degrees of regular automorphisms of the complex
space $\C^k$ (see  \cite{Sibony}) ; this idea was explained to us by an
anonymous referee (compare \cite{BC}). It may be considered as
an extension of a theorem due to Matsusaka and Mumford (see \cite{MaMu}, and \cite{KSC}, Exercise 5.6).

\begin{thm}\label{thm:sib}
Let $\k$ be a field. 
Let $M$ be a smooth connected projective variety defined over $\k$. Let $f$ 
be a birational transformation of $M$. Assume that the following two properties are
satisfied.
\begin{itemize}
\item[(i)] the Picard number of $M$ is equal to $1$;
\item[(ii)] the indeterminacy sets of $f$ and its inverse satisfy
\[
\dim(\Ind(f))+\dim(\Ind(f^{-1}))<\dim(M)-2.
\]
\end{itemize}
Then $f$ is an automorphism of $M$.
\end{thm}

Moreover, $\Aut(M)$ is an algebraic group because the Picard number of $M$ is equal to $1$. 
As explained below, this statement provides a direct proof of the following corollary, which was our initial motivation. 

\begin{cor}\label{cor:sib}
Let $M_0$ be a smooth, connected, projective variety with Picard number $1$. Let $m$ be a positive integer, and $\pi_i\colon M_{i+1}\to M_i$, $i=0, \ldots, m-1$, be a sequence of blow-ups of smooth irreducible subvarieties of dimension at most $r$. 
If $\dim(M_0)>2r+2$ then the number of connected components of $\Aut(M_m)$ is finite; moreover, the projection $\pi\colon M_m\to M_0$ conjugates $\Aut(M_m)$ to a subgroup of the algebraic group $\Aut(M_0)$. \end{cor}

For instance, if $M_0$ is the projective space (respectively a cubic hypersurface of $\P^4_{\k}$) and if one modifies $M_0$ by a finite sequence of blow-ups of points, then $\Aut(M_0)$ is isomorphic to a linear algebraic subgroup of $\PGL_4(\k)$ (respectively is finite). This provides a sharp (and strong) answer to a question of Eric Bedford. In Section \ref{par:CAut}, we provide a second, simpler proof of this last statement.

\begin{rem} The initial question of E. Bedford concerned the existence of automorphisms of compact K\"ahler manifolds with positive topological entropy in dimension $>2$. This link with dynamical systems is described, for instance, in \cite{Cantat:Pano}.
If a compact complex surface $S$ admits an automorphism with positive entropy, then $S$ is K\"ahler and is obtained from the projective plane $\P^2(\C)$, a torus,  a $K3$ surface or an Enriques surface, by a finite sequence of blow-ups (see \cite{C1,C2} and \cite{Nagata}). 
Examples of automorphisms with positive entropy are easily constructed on tori, K3 surfaces, or Enriques surfaces. 
Examples of automorphisms with positive entropy on rational surfaces are given in \cite{BK2,BKauto,McMullen}; these examples are obtained from 
birational transformations $f$ of the plane by a finite sequence of blow-ups that resolves all indeterminacies of $f$ and its iterates simultaneously. 
These results suggest to look for birational transformations of $\P^n_\C$, $n\geq 3$, that can be lifted to automorphisms with a nice dynamical 
behavior after a finite sequence of blow-ups; the above result shows that at least one center of the blow-ups must have dimension $\geq n/2-1$. 
\end{rem}

\begin{rem}
Recently, Tuyen Truong obtained results which are
similar to Corollary \ref{cor:sib}, but with hypothesis on the Hodge structure and nef classes
of $M_0$ that replace our strong hypothesis on the Picard number (see \cite{tuyen,tuyena}).
\end{rem}

\section*{Acknowledgement}
We received many interesting comments after we had posted a first version of this work on the web (see \cite{BC}). In particular, we
 thank Igor Dolgachev, Mattias Jonsson, Tuyen Truong for their valuable comments, and Brian Lehmann for stimulating correspondence. An anonymous referee explained to us how the main 
ideas of \cite{Sibony}, Prop. 2.3.2 and Rem. 2.3.3, could be applied in our setting 
to obtain Theorem 1.1 for $M=\P^k(\C)$, and  then recover the first results 
of \cite{BC}. We would like to express our gratitude to him/her.  
 
\section{Dimensions of Indeterminacy loci} 

In this section, we prove Theorem \ref{thm:sib} under a slightly more general assumption. Indeed, we replace assumption (i) with the following
assumption 
\begin{itemize}
\item[(i')] There exists an ample line bundle $L$ such that $f^*(L)\cong L^{\otimes d}$ for some $d>1.$  
\end{itemize}
This property is implied by (i). Indeed, if $M$ has Picard number $1$, the torsion-free part of the N\'eron-Severi group of $M$ is isomorphic to $\Z$, and is generated by the class $[H]$ of an ample divisor $H$. Thus, $[f^*H]$ must be a multiple of $[H]$.

In what follows, we assume that $f$ satisfies property (i') and property (ii). Replacing $H$ by a large enough multiple, we may and do assume
that $H$ is very ample. Thus, the complete linear system $\vert H\vert$ provides an embedding of $M$ into some  projective space $\P^n_\k$, and
we identify $M$ with its image in $\P^n_\k$. With such a convention, members of $\vert H\vert$ correspond to hyperplane sections of $M$.

\subsection{Degrees}  Denote by $k$ the dimension of $M$, and by $\deg(M)$ its degree, i.e. the number
of intersections of $M$ with a generic subspace of dimension $n-k$. 

If $H_1$, ..., $H_k$ are hyperplane sections of $M$, and if $f^*(H_1)$ denotes the total transform of $H_1$ under the action of $f$, one defines 
the degree of $f$ by the following intersection of divisors of $M$
\[
\deg(f)= \frac{1}{\deg(M)}f^*(H_1)\cdot H_2 \cdots H_k . 
\]
Since $M$ has Picard number $1$, we know that divisor class $[f^*(H_1)]$ is proportional to $[H]$. Our definition of $\deg(f)$ implies that 
$f^*[H_1]=\deg(f)[H_1]$. As a consequence, 
\[
f^*(H_1)\cdot f^*(H_2)\cdots f^*(H_j)\cdot H_{j+1} \cdots H_k = \deg(f)^j \deg(M)
\]
for all $0\leq j\leq k$.

\subsection{Degree bounds}\label{par:degbounds} 
Assume that the sum of the dimension of $\Ind(f)$ and of $\Ind(f^{-1})$ is at most $k-3$. Then there exist at least two integers
$l\geq 1$ such that 
\begin{eqnarray*}
\dim(\Ind(f)) & \leq & k-l-1;\\
\dim(\Ind(f^{-1})) & \leq & l-1.
\end{eqnarray*}
Let $H_1$, ..., $H_l$ and $H'_1$, ..., $H'_{k-l}$ be generic hyperplane sections of $M$; by Bertini's theorem, 
\begin{itemize}
\item[(a)] $H_1$, ..., $H_l$ intersect transversally the algebraic variety $\Ind(f^{-1})$ (in particular, $H_1\cap \ldots \cap H_l$ does
not intersect $\Ind(f^{-1})$ because $\dim(\Ind(f^{-1})) <  l$);
\item[(b)] $H'_1$, ..., $H'_{k-l}$ intersect transversally the algebraic variety $\Ind(f)$ (in particular, $H'_1\cap \ldots \cap H'_{k-l}$ does
not intersect $\Ind(f)$ because $\dim(\Ind(f)) < k- l$).
\end{itemize}
For $j\leq l$, consider the variety $V_j=f^*(H_1\cap \ldots \cap H_j)$: In the complement of $\Ind(f)$, $V_j$ is smooth, of dimension $k-j$; since 
$j\leq l$ and $\dim(\Ind(f))< k-l$, $V_j$ extends in a unique way as a subvariety of dimension $k-j$ in $M$. The varieties $V_j$ are reduced and
irreducible. 

Since each $H_i$, $1\leq i\leq l$, intersects $\Ind(f^{-1})$ transversally, $f^*(H_i)$ is an irreducible hypersurface (it does not contain any component 
of the exceptional locus of $f$). Thus 
\begin{eqnarray*}
V_j & = & f^*(H_1\cap \ldots \cap H_j)\\
& = & f^*(H_1)\cap \ldots \cap f^*(H_j)
\end{eqnarray*}
is the intersection of $j$ hypersurfaces of the same degree; for $j=l$ one gets 
\[
\deg(f)^l\deg(M)=f^*(H_1\cap \ldots \cap H_l)\cdot (H'_1\cap \cdots \cap H'_{k-l}).
\] 
More precisely, since the $H'_i$ are generic, this intersection is transversal and $V_j\cdot (H'_1\cap\ldots \cap H'_{k-l})$ is
made of $\deg(f)^l\deg(M)$ points, all of them with multiplicity $1$, all of them in the complement of $\Ind(f)$ (see property
(b) above). 

Similarly, one defines the subvarieties $V'_j=f_*(H'_1\cap \ldots H'_j)$ with $j\leq k-l$; as above, these subvarieties
have dimension $k-j$, are smooth in the complement of $\Ind(f^{-1})$, and uniquely extend to 
varieties of dimension $k-j$ through $\Ind(f^{-1})$. Each of them is equal to the intersection of the 
$j$ irreducible divisors $f_*(H_i)$, $1\leq i\leq j$. Hence,
\[
(H_1\cap \ldots \cap H_l)\cdot V'_{k-l} = \deg(f^{-1})^{k-l}\deg(M).
\]

If one applies the transformation $f\colon M\setminus \Ind(f)\to M$ to  $V_l$ and to $(H'_1\cap \cdots \cap H'_{k-l})$, one deduces that $\deg(f)^l\deg(M)\leq \deg(f^{-1})^{k-l}\deg(M)$, 
because all points of intersection of $V_l$ with $(H'_1\cap \ldots \cap H'_{k-l})$ are contained in the complement of $\Ind(f)$. 
Applied  to $f^{-1}$, the same argument provides the opposite inequality. Thus, 
\[
\deg(f)^l=\deg(f^{-1})^{k-l}
\]
Since there are at least two distinct values of $l$ for which this equation is satisfied, one concludes that 
\[
\deg(f)=\deg(f^{-1})=1.
\]
As a consequence, $f$ has degree $1$ if it satisfies  assumptions (i') and (ii), . 

\subsection{From birational transformations to automorphisms} To conclude the proof of Theorem \ref{thm:sib}, one applies the following lemma. 

\begin{lem}\label{lem:bir-to-aut}
Let $M$ be a smooth projective variety and $f$ a birational transformation of $M$. 
If there exists an ample divisor $H$ such that $f^*H$ and $f_*(H)$ are numerically equivalent to $H$, then
$f$ is an automorphism. 
\end{lem}

\begin{proof}
Taking multiples, we assume that $H$ is very ample.
Consider the graph $Z$ of $f$ in $M\times M$, together with its two natural projections $\pi_1$ and
$\pi_2$ onto $M$. 

The complete linear system $\vert H\vert$ is mapped by $f^*$ to a linear system $\vert H'\vert$ with the
same numerical class, and vice versa if one applies $f^{-1}$ to $\vert H'\vert$. Thus, $\vert H'\vert$ is also a complete
linear system, of the same dimension. Both of them are very ample (but they may differ if the dimension of $\Pic^0(M)$ is positive).

Assume that $\pi_2$ contracts a curve $C$ to a point $q$. 
Take a generic member $H_0$ of $\vert H\vert $: It does not intersect $q$, and $\pi_2^*(H_0)$ does not intersect $C$. 
The projection $(\pi_1)_*(\pi_2^*(H_0))$ is equal to $f^*(H_0)$; since $f^*$ maps the complete linear system 
$\vert H\vert$ to the complete linear system $\vert H'\vert$ and $H_0$ is generic, we may assume that $f^*(H_0)$ is a generic member of $\vert H'\vert$.
As such, it does not intersect the finite set $\pi_1(C)\cap\Ind(f)$. Thus, there is no fiber of $\pi_1$ that intersects simultaneously $C$ and $(\pi_2)^*(H_0)$, and $(\pi_1)_*(\pi_2^*(H_0))$ does not intersect $C$. This contradicts the fact that $f^*(H_0)$ is ample.
\end{proof}

\subsection{Conclusion, and K\"ahler manifolds} Under the assumptions of Theorem~\ref{thm:sib}, Section \ref{par:degbounds} shows that $f^*H$ is 
numerically equivalent to $H$. Lemma~\ref{lem:bir-to-aut} implies that $f$ is an automorphism. This concludes the proof of 
Theorem~\ref{thm:sib}. 

This proof is inspired by an argument of Sibony in \cite{Sibony} (see Proposition 2.3.2 and Remark 2.3.3); which makes use of complex
analysis: the theory of closed positive current, and intersection theory. With this viewpoint, one gets the following statement. 

\begin{thm}
Let $M$ be a compact K\"ahler manifold and $f$ a bi-meromorphic transformation of $M$. 
Assume that 
\begin{itemize}
\item[(i)] there exists a K\"ahler form $\omega$ such that the cohomology class of $f^*\omega$ is proportional to the cohomology class of $\omega$;
\item[(ii)] the indeterminacy locus of $f$ and its inverse satisfy
\[
\dim(\Ind(f))+\dim(\Ind(f^{-1}))<\dim(M)-2.
\]
\end{itemize}
Then $f$ is an automorphism of $M$ that fixes the cohomology class of $\omega$.
\end{thm}
Moreover, Lieberman's theorem (see \cite{Lieberman}) implies that a positive iterate $f^m$ of $f$ is contained in the connected
component of the identity of the complex Lie group $\Aut(M)$. 

\subsection{Proof of Corollary \ref{cor:sib}}  Since $M_m$ is obtained from $M_0$ by a sequence of blow-ups of centers
of dimension $<\dim(M_m)/2-1$, all automorphisms $f$ of $M_m$ are conjugate, through the obvious birational morphism $\pi\colon M_m\to M_0$,
to birational transformations of $M_0$ that satisfy
\[
\dim(\Ind(f))<\dim(M_0)/2-1 \; {\text{and}} \; \dim(\Ind(f^{-1}))<\dim(M_0)/2-1.
\]
Thus, by Theorem~\ref{thm:sib} $\pi$ conjugates $\Aut(M)$ to a subgroup of $\Aut(M_0)$. Moreover, given any polarization of $M_0$ by a  very ample class, all elements of $\Aut(M_0)$ have degree $1$ with respect to this polarization. Hence, $\Aut(M_0)$ is an algebraic group, and the kernel of the
action of $\Aut(M_0)$ on $\Pic^0(M_0)$ is a linear algebraic group; if $\Pic^0(M_0)$ is trivial, there is a projective embedding of $\Theta\colon M_0\to
\P^n_\k$ that conjugates $\Aut(M_0)$ to the group of linear projective transformations $G\subset \PGL_{n+1}(\k)$ that preserve $\Theta(M)$.

\section{Constraints on automorphisms from the structure of the intersection form} \label{par:CAut}
Let $X$ be a smooth projective variety of dimension $k$ over a field $\k$. Denote by $\NS(X)$ the N\'eron-Severi group of $X$, {\sl{i.e.}} the group of classes of divisors for the numerical equivalence relation. We consider the multi-linear forms 
$$Q_d\colon \NS(X)^d\to \Z$$ which are defined by
\[
Q_d(u_1,u_2,\dots,u_d)=u_1\cdot u_2 \cdots u_d\cdot K_X^{k-d}.
\]
These forms are invariant under $\Aut(X)^*$ and we shall derive new constraints on the size of $\Aut(X)^*$ from this invariance.

\begin{thm} \label{smooth}  
 Let $X$ be a smooth projective variety of dimension $k\geq 3$, defined over a field $\k$. Let $d$ be an integer that satisfies $3\leq d\leq k$. If the projective variety 
 \[
 W_d(X):=\{u\in \P(\NS(X)\otimes_{\Z}\C)|\ Q_d(u,u,\dots,u)=0\}
 \]
 is smooth, then $\Aut(X)^*$ is finite.
 \end{thm}
 \begin{proof}
 The group $\Aut(X)^*$ acts by linear projective transformations on the projective space $\P(\NS(X)\otimes_{\Z}\C)$ and preserves the smooth hypersurface $W_d.$ Since $d\geq 3$ it follows from \cite{MM} that the group of linear projective transformations preserving a smooth hypersurface of degree $d$ is finite. Hence, there is a finite index subgroup $A$ of $\Aut(X)^*$ which is contained in the center of $\GL(\NS(X))$; since the later is a finite group of homotheties, this finishes the proof.
 \end{proof}
 
As a corollary, let us state the following one, already obtained in the previous sections:
 \begin{cor}
 Let $X$ be a smooth projective variety of dimension $k\geq 3.$ Assume that there exists a birational morphism 
 $\pi:X\to V$ such that 
 \begin{itemize}
 \item the Picard number of $V$ is equal to 1
 \item $\pi^{-1}$ is the blow-up of $l$ distinct points of $V.$
 \end{itemize}
 
 Then $\Aut(X)^*$ is a finite group.
\end{cor}
\begin{proof}
We identify $\NS(V)$ with $\Z e_0$ where $e_0$ is the class of an ample divisor. Let $a:=e_0^k.$ Since $X$ is obtained from $V$ by blowing up $l$ distinct points $p_1,\dots,p_l$ we have 
\[
\NS(X)=\Z e_0+\bigoplus_{1\leq i\leq l}\Z e_i
\]
where $e_i$ is the class of the exceptional divisor $E_i:=\pi^{-1}(p_i).$ Then the form $Q_k$ is given by 
\[
Q_k(u)=a(X_0)^k+(-1)^{k+1}\sum_{i=1}^l(X_i)^k
\]
where $u=X_0e_0+\sum_i X_ie_i$ and $[X_0:\dots X_l]$ denotes the homogeneous coordinates on $\P(\NS(X)\otimes_{\Z}\C).$ 
Hence, the projective variety defined by $Q_k$ in $\P(\NS(X)\otimes_{\Z}\C)$ is smooth and $\Aut(X)^*$ is finite. 
\end{proof}



\end{document}